\documentclass[12pt]{amsart}
\usepackage{amsthm,amsmath}
\usepackage{amssymb}
\input xypic

\newcommand{\CC}{{\mathbb{C}}}

\newcommand{\EE}{{\mathbb{E}}}
\newcommand{\PP}{{\mathbb{P}}}
\newcommand{\QQ}{{\mathbb{Q}}}

\newcommand{\ZZ}{{\mathbb{Z}}}

\newcommand{\calO}{{\mathcal O}}

\newcommand{\calA}{{\mathcal A}}

\newcommand{\calM}{{\mathcal M}}

\newcommand{\calP}{{\mathcal P}}

\newcommand{\calX}{{\mathcal X}}

\newcommand{\op}{\operatorname}

\newcommand{\Perf}{{\calA_g^{\op {Perf}}}}

\newcommand{\Pic}{\op{Pic}}
\newcommand{\Jac}{\op{Jac}}

\newcommand\oMg{\overline{\calM}_g}
\newcommand\oAg{\overline{\calA}_g}

\newcommand\T{\Theta}
\newcommand\ud{{\underline{d}}}
\theoremstyle{plain}
\newtheorem{thm}{Theorem}[section]

\newtheorem{prop}[thm]{Proposition}

\theoremstyle{definition}

\newtheorem{rem}[thm]{Remark}
\newtheorem{ex}[thm]{Example}
\begin{document}
\title[The universal semiabelian variety]{The zero section of the universal semiabelian variety, and the double ramification cycle}
\author{Samuel Grushevsky}
\address{Mathematics Department, Stony Brook University, Stony Brook, NY 11790-3651, USA}
\email{sam@math.sunysb.edu}
\thanks{Research of the first author supported in part by National Science Foundation under the grants DMS-10-53313/12-01369.}
\author{Dmitry Zakharov}
\address{Mathematics Department, Stony Brook University, Stony Brook, NY 11790-3651, USA}
\email{dvzakharov@gmail.com}
\begin{abstract}
We study the Chow ring of the boundary of the partial compactification of the universal family of principally polarized abelian varieties. We describe the subring generated by divisor classes, and compute the class of the partial compactification of the universal zero section, which turns out to lie in this subring. Our formula extends the results for the zero section of the universal uncompactified family.

The partial compactification of the universal family of ppav can be thought of as the first two boundary strata in any toroidal compactification of $\calA_g$. Our formula provides a first step in a program to understand the Chow groups of $\overline{\calA}_g$, especially of the perfect cone compactification, by induction on genus. By restricting to the image of $\calM_g$ under the Torelli map, our results extend the results of Hain on the double ramification cycle, answering Eliashberg's question.
\end{abstract}
\maketitle

\section*{Introduction}
We are mainly interested in the Chow and cohomology groups of (compactified) moduli spaces of principally polarized abelian varieties (ppav). The tautological ring of the moduli space of ppav $\calA_g$ is defined as the subring $R^*(\calA_g)\subset A^*(\calA_g)$ of the Chow ring (or of cohomology ring $RH^*(\calA_g)\subset H^*(\calA_g)$) generated by the Chern classes $\lambda_i:=c_i(\EE)$ of the rank $g$ Hodge bundle $\EE\to\calA_g$, with fiber over $A$ being $H^{1,0}(\CC)$. Unlike the case of curves, this tautological ring is known completely. For a suitable toroidal compactification $\oAg$ van der Geer \cite{vdgeercycles} proved that $RH^*(\oAg)$ is generated by the $\lambda_i$ with the only relations being the homogeneous degree pieces of the basic relation
\begin{equation}\label{basic}
(1+\lambda_1+\cdots+\lambda_g)(1-\lambda_1+\cdots+(-1)^g\lambda_g)=1.
\end{equation}
In cohomology this relation follows from  the triviality of $\EE\oplus\overline{\EE}$, the total space of the bundle of first cohomology. Esnault and Viehweg \cite{esvi} proved the much more delicate result that it also holds in the Chow ring of $\oAg$, which implies that $R^*(\oAg)=RH^*(\oAg)$.

Furthermore, van der Geer \cite{vdgeercycles} also proved that the tautological ring $R^*(\calA_g)=RH^*(\calA_g)$ is obtained from $R^*(\oAg)$ by imposing one more relation $\lambda_g=0$. Thus in $R^*(\oAg)$ the class $\lambda_g$ can be represented by a cycle supported on the boundary, and it is a natural question to find a suitable representative for it. A lot of progress on this was made by Ekedahl and van der Geer \cite{ekvdgcycles},\cite{ekvdgorder}. In particular, in characteristic $p$ suitable cycles were constructed, but over $\CC$ this question remains open. Note that in characteristic zero Keel and Sadun \cite{kesa} proved Oort's conjecture that $\calA_g$ does  not have complete subvarieties of codimension $g$.

One naturally defined geometric locus in $\oAg$ is the locus $\delta_g$, the closure of the locus of trivial extensions of semiabelic varieties of torus rank one. This locus was introduced and studied by Ekedahl and van der Geer \cite{ekvdgcycles}. We denote $\calA_g'\supset\calA_g$ Mumford's partial compactification, obtained by adding semiabelic varieties of torus rank one (compactifications of $\CC^*$-extensions of $(g-1)$-dimensional ppav). The boundary $\calA_g'\setminus\calA_g$ is then the universal family of $(g-1)$-dimensional Kummer varieties (quotients of ppav by the $-1$ involution), and admits the zero section. The class $\delta_g$ is defined to be the class of the closure of the image of the zero section in a suitable toroidal compactification $\oAg$ (recall that all toroidal compactifications contain $\calA_g'$). Ekedahl and van der Geer show that on $\oAg$ the class $\lambda_g$ is equal to $(-1)^g\zeta(1-2g)\delta_g$ up to classes supported deeper in the boundary, on $\oAg\setminus\calA_g'$, in other words that on $\calA_g'$ the class $\lambda_g$ is proportional to $\delta_g$. Thus understanding the class $\delta_g$ could lead to finding an explicit geometric cycle representing $\lambda_g$ in characteristic zero. A study of the locus $\delta_g$ is also natural since Shepherd-Barron \cite{shepherdbarron} showed that in the perfect cone toroidal compactification $\Perf$ the normalization of the closure of the zero section is equal to $\calA_{g-1}^{\op{Perf}}$. Thus a full understanding of the class $\delta_g$ could provide an inductive approach for understanding the cohomology of the perfect cone compactification, for example addressing the conjecture of Erdenberger, Hulek, and the first author \cite{ergrhu2} on intersection numbers of divisors on $\calA_{g-1}^{\op{Perf}}$. We note that for $g\le 3$ the locus $\delta_g$ was fully described, and its class in $\overline{\calA}_g$ was computed completely by van der Geer in \cite{vdgeerchowa3}, but nothing was previously known for higher $g$.

Denote $\calX_g\to\calA_g$ the universal family of ppav, and by $\calX_g'\to\calA_g'$ its partial compactification (note that the existence of a universal compactified family $\overline{\calX}_g$ over a full toroidal compactification $\overline{\calA}_g$ is only known for the second Voronoi toroidal compactification by the work of Alexeev \cite{alexeev}). Our main result is the computation of the class of the closure of the zero section $z_g':\calA_g'\to\calX_g'$, which turns out to be a polynomial in divisor classes and a natural codimension two ``gluing locus'' in $\calX_g'$, see Theorem \ref{thm:main}. Moreover, we prove that the divisor classes and this codimension two class generate a certain geometric subring of the Chow ring of $\calX_g'$, see Theorem \ref{thm:expressible}. We also describe the algebraic cohomology of the universal family $\calX_g\times_{\calA_g}\calX_g$ of products of ppav, see Theorem \ref{thm:subring}. Since $\calX_{g-1}$ is a cover of the boundary of $\calA_g'$, our results mean that on $\overline{\calA}_g$ we compute the class $\delta_g$ up to the stratum parameterizing semiabelic varieties of torus rank two.

\smallskip
Our results also have consequences for the moduli space of curves $\calM_g$. Let $\overline{\calM}_{g,n}$ denote the Deligne--Mumford compactification by stable marked curves, and let $\calM_{g,n}^{ct}$ denote the partial compactification by stable curves of compact type. Recall that the tautological rings of $R^*(\overline{\calM}_{g,n})\subset A^*(\overline{\calM}_{g,n})$ are defined as the smallest collection of subrings closed under the natural gluing and forgetful maps, and the tautological rings of $\calM_{g,n}$ and $\calM_{g,n}^{ct}$ are defined by restriction. The tautological ring of $\calM_g$ is generated by the Mumford-Morita-Miller classes $\kappa_i$, while the tautological rings of the compactifications contain additional classes coming from the boundary.
Faber's conjecture \cite{faberconjecture} states that the tautological ring of $\calM_g$ (respectively of $\calM_{g,n}^{ct}$ and $\overline{\calM}_{g,n}$) is Gorenstein with socle in dimension $g-2$ (respectively $2g-3+n$ and $3g-3+n$) --- such rings are also called Poincar\'e duality rings. That is to say that the tautological ring is zero above the socle dimension, one-dimensional in the socle dimension, and has perfect pairing to the socle dimension. The vanishing and one-dimensionality are known for all tautological rings (see \cite{loo},\cite{faberonedim},\cite{ionel},\cite{grva}, the perfect pairing statement is currently not known to hold for $\calM_g$ when $g\geq 24$, for $\calM_g^{ct}$ when $g\geq 6$, and was
recently shown to be false for $\overline{\calM}_{2,n}$ for a suitable $n$ by Petersen and Tommasi \cite{2012PetersenTommasi}. We refer to the work of  Pandharipande, Pixton and Zvonkine \cite{2013PandharipandePixtonZvonkine} for the recent progress on understanding the relations in the tautological ring, and further review of the state of the art.

Another interesting question is whether classes of various naturally defined geometric loci in $\calM_g$ are tautological, and whether the classes of their closures in $\oMg$ are tautological. For tautological classes on $\oMg$ that vanish in $\calM_g$ one can also ask to find explicit geometric representatives --- in particular this question is of interest for the class $\lambda_g$, which is a pullback from the moduli space of ppav.

Our results on $\calX_g'$ yield a further understanding of a natural codimension $g$ class on $\overline{\calM}_{g,n}$: the two-branch-point locus, also called the double ramification cycle, which is formally defined and discussed in detail by Faber, Shadrin, and Zvonkine \cite{FSZ} (while it seems to have been first considered by Ionel, see \cite{ionel}).
It is defined as the closure of the locus of $(X,p_1,\ldots,p_n)\in\calM_{g,n}$ such that a linear combination $\sum d_i p_i$ is a principal divisor on $X$ (for some fixed $d_i\in\ZZ, \sum d_i=0$). This is a natural ``double Hurwitz'' locus of curves admitting a map to $\PP^1$ with prescribed preimages and ramification at 0 and $\infty$. Its class is also of interest in Gromov--Witten theory (see \cite{FSZ} for details), and the question of computing it is due to Eliashberg. The class of the closure of this locus in $\calM_{g,n}^{ct}$ was recently computed by Hain \cite{hainnormal}. We extend his computation further into the boundary of $\overline{\calM}_{g,n}$, to the open subset parameterizing curves with at most one non-separating node, by pulling back the class of the locus $\delta_g$.

\smallskip
While for $\oMg$ the classes of the boundary divisors, and possibly the classes of the closures of various geometric loci, are tautological by the work of Faber and Pandharipande \cite{fapabntaut}, for $\oAg$ already the boundary divisor(s) are non-tautological (as their classes are clearly not proportional to $\lambda_1$). Thus defining a suitable ``extended'' tautological subring of $A^*(\oAg,\QQ)$ is a natural central further question to study; one could hope that such a ring would be defined geometrically, and would contain the classes of geometrically defined loci. Some results in this direction were obtained by Hulek and the first author in \cite{grhu1}, but the situation is far from clear, and studying natural geometric loci in $\oAg$ is thus of particular interest.

Note also that since the Deligne-Mumford compactification $\oMg$ admits a morphism to the second Voronoi (by the work of Namikawa \cite{namikawabook}) and to the perfect cone (by the work of Alexeev and Brunyate \cite{albr}) toroidal compactifications of $\calA_g$ (and the image is the same, landing in the matroidal locus by Melo and Viviani \cite{mevi}), restricting a geometric cycle representing $\lambda_g$ on $\oAg$ to the image of the Torelli map would allow one to relate the tautological rings of $\oMg$ and $\calM_g^{ct}$ (which is the preimage of $\calA_g$ under the Torelli map, as a stack), and perhaps to obtain a direct computational proof of the $\lambda_g$-conjecture (proven by Faber and Pandharipande \cite{fapalambdag}, to which we also refer for a discussion).

\section{Statement of results}
The principal result of our paper is the computation of the class of the closure of the zero section of the universal abelian variety in the partial compactification $\calX_g'\to\calA_g'$.

\begin{thm}\label{thm:main}
Let $\calX_g'\to\calA_g'$ denote the partial compactification of the universal family of ppav $\calX_g\to\calA_g$, let $z_g:\calA_g\to\calX_g$ denote the zero section, let $z_g':\calA_g'\to\calX_g'$ denote the closure of the zero section in the partial compactification, and let $Z_g'$ denote its class in $A^{g}(\calX_g',\QQ)$. Then we have
\begin{equation}\label{eqn:main}
  Z_g'=\displaystyle\sum_{a+b+2c=g} \alpha_{a,b,c} (\Theta-D/8)^a D^b(\Delta-2\Theta D)^c,
\end{equation}
where the positive coefficients $\alpha_{a,b,c}$ are given by
\begin{equation}\label{eqn:beta}
  \alpha_{a,b,c}= \frac{(-1)^{b+c+1}(2^{-b-c}-2^{1-3b-3c})(2a+2b+2c-1)!!B_{2b+2c}}{(2a+2c-1)!!(2b+2c-1)!!a!b!c!}.
\end{equation}
Here $B_n$ denotes the $n$-th Bernoulli number, $\Theta\in A^1(\calX_g',\QQ)$ denotes the class of the universal theta divisor trivialized along the zero section, $D\in A^1(\calX_g',\QQ)$ denotes the class of the boundary $\calX_g'\setminus\calX_g$, and $\Delta\in A^2(\calX_g',\QQ)$ denotes the class of the gluing locus within $D$, where the $0$ and $\infty$ sections of the universal Poincar\'e bundle that is the total space of $\calX_g'\setminus\calX_g$ are identified (all considered non-stacky, see below for details).
\end{thm}
\begin{rem}
The classes $\Theta-D/8$ and $\Delta -2\Theta D $ above may  seem like a random choice, but in fact have a geometric significance. Indeed, $\Theta-D/8$  is in a sense the class of the theta divisor, with generic vanishing on the boundary taken out, and appears for example in Grothendieck-Riemann-Roch computations in \cite{ekvdgcycles}, while $\Delta -2\Theta D $ is a natural ``shift-invariant'' class (see below).

Equivalently, the class of the partial compactification of the zero section can be written as
$$
  Z_g'=\displaystyle\sum_{a+b+2c=g}\eta_{a,b,c}\Theta^a D^b\Delta^c,
$$
where the coefficients $\eta_{a,b,c}$ are equal to
$$
\frac{(-1)^{b+c}(2c+2b-1)!!}{2^{3b+3c}a!c!}\displaystyle\sum\limits_{x=0}^b
 \frac{(2-2^{2c+2x})B_{2c+2x}}{(2c+2b-2x-1)!!(2c+2x-1)!!(b-x)!x!}.
$$
\end{rem}

We note that as $\calX_{g}/\pm 1$ is the boundary of the partial compactification $\calA_{g+1}'$, we can interpret the above result as computing the class $\delta_{g+1}\in A^*(\overline{\calA}_{g+1},\QQ)$ up to the second boundary stratum, of semiabelic varieties of torus rank two --- see Remark~\ref{rem:goingtoAg} for more details on this.

The theorem above was surprising to us, as it claims that $Z_g'$, which is a degree $g$ class, admits a polynomial expression in classes of degree $1$ and $2$. However, this turns out to be a fairly general phenomenon. Namely, we prove the following result.

\begin{thm}\label{thm:expressible}
Let $\widetilde{Y}$ denote the normalization of the boundary of the partial compactification $\calX_g'\to\calA_g'$. Any class in $A^*(\calX_g',\QQ)$ whose pullback to $\widetilde{Y}$ is a polynomial in divisor classes on $\widetilde{Y}$ can be expressed on $\calX_g'$ as a polynomial in the three classes $\Theta$, $D$ and $\Delta$.
\end{thm}

Along the way of proving these results, we also further investigate the geometry and intersection theory of the total space of the universal Poincar\'e bundle (i.e.~of $\calX_g'\setminus\calX_g$), which may be of independent interest.

\smallskip
Turning to the moduli space of curves, we apply the theorem above to obtain a partial answer to the following question of Eliashberg. Let $\ud=(d_1,\ldots,d_n)\in \ZZ^n$ be integers summing to zero, and consider the locus $R_\ud$ of curves $(X,p_1,\ldots,p_n)\in \calM_{g,n}$ such that $\sum d_ip_i$ is a principal divisor on $X$. This locus is known as the {\it double ramification cycle}, and the question is to compute the class of its closure in $\overline{\calM}_{g,n}$.

On a smooth curve $X$, the divisor $\sum d_i p_i$ is principal if and only if its image in $\op{Jac}(X)$ is zero. Therefore, the double ramification cycle on $\calM_{g,n}$ can be computed by pulling back the zero section of the universal Jacobian under the Abel--Jacobi map $s_{\ud}:\calM_{g,n}\to \op{Jac}$ that sends $(X,p_1,\ldots,p_n)$ to $\mathcal{O}_X(\sum d_i p_i)\in \op{Jac}(X)$. This map naturally extends to curves of compact type, since the Jacobians of such curves are abelian varieties, and this approach was used by Hain \cite{hainnormal} to compute the class of the closure of $R_\ud$ in $\calM_{g,n}^{ct}$.

In this paper, we take this approach one step further. To extend the Abel--Jacobi map beyond $\calM_{g,n}^{ct}$, we need to allow the target abelian varieties to degenerate. The Torelli embedding $\calM_g\rightarrow\calA_g$ extends to a morphism $\oMg\to\oAg$ both to the perfect cone and the second Voronoi compactification, by \cite{namikawabook}, \cite{albr}, \cite{mevi}. The preimage of $\calA_g'$ under the Torelli morphism to any toroidal compactification is equal to the locus of stable curves of geometric genus $g-1$. The Abel--Jacobi map does not extend to this locus -- in examples \ref{fails1} and \ref{fails2}, we show that the Abel--Jacobi map cannot in general be defined for curves whose dual graph has a non-trivial cycle having more than one edge. However, the Abel--Jacobi map does extend to the open locus of stable curves having at most one non-separating node (equivalently, whose dual graph has at most one loop consisting of a single edge). Denoting this locus $\overline{\calM}_{g,n}^o$, we obtain a morphism $\overline{\calM}_{g,n}^o\to\calX_g'$. Computing the class of the zero section in the partial compactification and pulling it back, we find the class of the closure of the double ramification cycle in $\overline{\calM}_{g,n}^o$.

\begin{thm}\label{thm:Mg}
Let $\overline{\calM}_{g,n}^o$ be the open subset of $\overline{\calM}_{g,n}$ parameterizing curves with at most one non-separating node. Let $\ud=(d_1,\ldots,d_n)\in \ZZ^n$ be integers summing to zero, and let $R_\ud$ denote the double ramification cycle defined above. Then the class of the closure $\overline{R}_\ud^o$ of  $R_\ud$ in $\overline{\calM}_{g,n}^o$ is equal in $A^g(\overline{\calM}_{g,n}^o,\QQ)$ to
$$
[\overline{R}_\ud^o]=\displaystyle\sum_{a+b=g}\eta_{a,b,0}(s^*_\ud\Theta)^a \delta_{irr}^b.
$$
where $\eta_{a,b,0}$ are the same as in Theorem~\ref{thm:main}, and $s^*_\ud\Theta$ denotes the pullback of the class $\Theta$ from $\calX_g'$ to $\overline{\calM}^o_{g,n}$ under the Abel--Jacobi map, computed in \cite{hainnormal}, \cite{grza1} to be
$$
s^*_\ud\Theta=\frac{1}{2}\displaystyle\sum_{i=1}^n d_i^2 K_i-\frac{1}{2}\displaystyle\sum_{P\subseteq I}
\left(d_P^2-\displaystyle\sum_{i\in P}d_i^2\right)\delta_0^P-\frac{1}{2}
\displaystyle\sum_{h> 0,P\subseteq I} d_P^2
\delta_h^P,
$$
where $K_i$ and $\delta_h^P$ are the standard divisor classes on $\overline{\calM}_{g,n}^o$ (see Section~\ref{sec:eliashberg} for details), $I=\{1,\ldots,n\}$ is the indexing set, and $d_P=\sum_{i\in P} d_i$.
\end{thm}

\begin{rem}
On the moduli space of curves of compact type this formula restricts to the result of Hain \cite{hainnormal} by taking only the constant term in $\delta_{irr}$, while on the moduli space of curves with rational tails (having a smooth component of maximum genus) this formula restricts to the result of Cavalieri, Marcus, and Wise \cite{cavalieri}.

We would like to stress, however, that while $\overline{\calM}_{g,n}^o\setminus\calM_{g,n}^{ct}$ is an irreducible divisor, computing a codimension $g$ class on $\overline{\calM}_{g,n}^o$ involves much more than computing it on $\calM_{g,n}^{ct}$, and then computing one extra coefficient.
\end{rem}
\begin{rem}
It may seem surprising that the degree 2 class $\Delta$ from theorem \ref{thm:main} does not appear here, i.e.~that we are only taking the constant term in $\Delta$ (that is, $c=0$) of the expression there. This is due to the fact that the image of $\overline{\calM}_{g,n}^o$ under the Abel-Jacobi map is disjoint from $\Delta$, as we will see in the proof of the theorem
\end{rem}

\medskip
The structure of the paper is as follows. In Section~\ref{sec:notation} we introduce the notation, and review the known results on the geometric structure of the boundary of $\calX_g'$ (which is also the second stratum of the boundary of $\overline{\calA}_{g+1}$), mostly following \cite{ergrhu2}. In Section~\ref{sec:semiabelic} we study the subring of its Chow ring generated by the divisor classes. In Section~\ref{sec:shift} we study the normalization of the boundary of $\calX_g'$ and describe the classes on the normalization that glue to classes on the actual boundary of $\calX_g'$, culminating with a proof of Theorem~\ref{thm:expressible}. In Section~\ref{sec:main} we study the closure of the zero section and obtain an expression for it, proving Theorem~\ref{thm:main}.
Finally, in Section~\ref{sec:eliashberg} we use this theorem, together with standard intersection techniques on $\overline{\calM}_{g,n}$, to obtain an answer to Eliashberg's problem, proving Theorem~\ref{thm:Mg}.

\section{Notation and known results}\label{sec:notation}

Throughout the text, we work with Chow groups with rational coefficients. The spaces that we work with are smooth Deligne-Mumford stacks, and thus the Chow groups admit a ring structure (below, we specifically avoid working with the Chow groups of the non-normal boundary $Y=\calX_g'\backslash\calX_g$).

We denote by $\calX_g\to\calA_g$ the universal family of principally polarized abelian varieties (considered as a stack). We will also be concerned with the universal family of Kummer varieties $\pi:\calX_g/\pm 1\to\calA_g$ -- the quotient of the universal family of ppav by the involution $\pm1$. We denote by $z_g:\calA_g\to\calX_g$ the zero section, and by abuse of notation, we also denote $z_g$ the image of the zero section as a locus in $\calX_g$ (or in $\calX_g/\pm 1$ depending on context). We denote $Z_g$ the class of the zero section in the Chow group $A^g(\calX_g)$. We denote by $T\in A^1(\calX_g)$ the class of a universal symmetric theta divisor trivialized along the zero section; note that such a divisor is only defined up to a translation by a 2-torsion point, but its class in the rational Chow ring is well-defined.

Our problem is motivated by the following result:

\begin{thm}[\cite{hainnormal} in homology, implied by the results of \cite{denmur} in Chow]\label{thm:0g}
The class of the zero section in $A^{g}(\calX_g)$ and in $H^{2g}(\calX_g)$ is equal to
$$
Z_g=\frac{T^g}{g!}.
$$
\end{thm}
\begin{rem}\label{HA}
This result has a long history, and many approaches to it have been developed. We are grateful to Richard Hain, Claire Voisin, and Gerard van der Geer for discussions on these topics. Indeed, Hain \cite[Prop.~8.1]{hainnormal} proves this result in cohomology using Hodge-theoretic methods, while the argument in the Chow ring uses the Fourier transform on the Chow ring, and is based on ideas of Deninger and Murre, including \cite[Cor.~2.22]{denmur}; this statement is given as \cite[Exercise~13.2]{vdgmoo}. We also refer to Section \ref{sec:semiabelic} for more results and a discussion of the relationship of the Poincar\'e bundle and the class $T$.
\end{rem}

The goal of this paper is to extend this formula to Mumford's partial compactification of the moduli space of ppav, which we denote by $\calA_g'$. In this section, we recall the construction of the universal family over the partial compactification.

The partial compactification is the blow-up of the partial Satake compactification $\calA_g\sqcup\calA_{g-1}$ along the boundary. The boundary of the partial compactification is the universal family $\calX_{g-1}$:
$$
\calA_g'=\calA_g\sqcup \calX_{g-1}.
$$

Geometrically, the boundary of the partial compactification parameterizes semiabelic varieties of torus rank one, described as follows. For a point $(B,b)\in \calX_{g-1}$, where $B\in \calA_{g-1}$ is an abelian variety of dimension $g-1$ and $b\in B$ a point on it, up to sign, the semiabelic variety corresponding to $(B,b)$ is obtained by compactifying the $\CC^*$-extension of $B$
$$
1\to \CC^*\to G\to B\to 0
$$
to a $\PP^1$-bundle $\widetilde{G}$ over $B$ by adding the $0$- and $\infty$-sections, and then gluing these sections with a shift by $b$ to obtain the non-normal variety $\overline{G}=\widetilde{G}/(\beta,0)\sim(\beta+b,\infty)$ (we use $\beta$ instead of the more standard notation $z$, to distinguish this from the zero section, and to emphasize that $\beta$ and $b$ are in a sense points of dual abelian varieties).

We extend the universal family $\pi:\calX_g\to\calA_g$ to a family over the partial compactification $\pi':\calX_g'\to\calA_g'$ by globalizing the construction above. We follow the notation of \cite{ergrhu2}, the results and setup of which we now recall. We let $\calX^2_{g-1}=\calX_{g-1}\times_{\calA_{g-1}} \calX_{g-1}$ be the fiberwise square, with $pr_i:\calX^2_{g-1}\to\calX_{g-1}$ denoting the projections to the two factors. Let $\calP$ denote the Poincar\'e bundle on $\calX_{g-1}^2$, and let $\widetilde{Y}=\PP(\calP\oplus\calO)$ denote the projectivization of $\calP$. We now define the extension $Y$ of the universal family over the boundary by gluing the $0$- and $\infty$-sections of $\widetilde{Y}$ with a shift by the second coordinate, and factorizing by the involution. In other words, we glue  $(B,\beta,b,0)\in \widetilde{Y}$ and $(B,\beta+b,b,\infty)\in \widetilde{Y}$, and then factorize by $j$, where $j$ denotes the involution on the semiabelic variety fiber of $\widetilde Y\to\calX_{g-1}$. We denote $\Delta\subset Y$ the gluing locus (and by abuse of notation its class in cohomology), i.e.~$\Delta$ denotes the image of the glued $0$- and $\infty$-sections of $\widetilde Y$. We summarize the geometry in the following diagram:
$$
\xymatrix{&&&(\widetilde{Y}=\PP(\calP\oplus\calO))/j\ar[ld]\ar[d]\ni(B,\beta,b,x)\\
\calX_g/\pm 1\ar[d]&\sqcup&Y\ar[d]&\calX_{g-1}^2\ni(B,\beta,b)\ar[ld]^{pr_2}\ar[d]^{pr_1}\\
\calA_g\ar[d]&\sqcup&\calX_{g-1}\ar[d]\ni (B,b)&\calX_{g-1}\ar[ld]\ni(B,\beta)\\
\calA_g&\sqcup & \calA_{g-1}\ni(B)
}
$$

We avoid working directly on the boundary family $Y$, because it is not normal and the Chow groups do not have an intersection product. We instead do all our computations on $\widetilde Y$, which is a $\PP^1$-bundle over $\calX_{g-1}^2$, and then only at the end take the involution and the gluing into account by requiring our computations to be invariant under them.

We now summarize the known results about the Chow rings of the various objects in the diagram.

The Picard group of $\calA_{g-1}$ is equal to the first Chow group and is generated by the first Chern class $\lambda_1$ of the Hodge bundle. The Picard group and the first Chow group of the universal family is $\Pic_\QQ(\calX_{g-1})=\QQ\lambda_1\oplus\QQ T$, where we recall that $T$ is the class of a universal symmetric theta divisor trivialized along the zero section.

The Picard group of the product family $\calX_{g-1}^2$ is generated  by the pullback of $\lambda_1$, which we denote by $L$, by the pullbacks $T_i=pr_i^*T$ of the theta divisors from the two factors, and by the class $P$ of the universal Poincar\'e bundle, also trivialized along the zero section (see \cite{ergrhu2}). By abuse of notation, we also use $L$, $T_1$, $P$ and $T_2$ to denote the pullbacks of these classes to $A^1(\widetilde{Y})$. We recall that by the results of Deninger and Murre \cite{denmur} (see also \cite{voisin}), the direct image $R\pi_*\QQ$ of a constant sheaf in any family of ppav admits a multiplicative decomposition. It follows, (see \cite[Prop.~4.3.6, Cor.~4.3.9]{voisinnotes} and Remark~\ref{HA}), that for classes $T_1$, $P$ and $T_2$ on $\calX_{g-1}^2$, all trivialized along the zero section by definition, a polynomial relation $f(T_1,P,T_2)=0$ holds in $H^*(\calX_{g-1}^2)$ if and only if it holds in the Chow ring and if and only if it holds fiberwise. Along the way of our computation, we compute the relations between these classes on a very general ppav, and thus describe entirely the subring of $H^*(\calX_{g-1}^2)$ (and of the Chow ring) generated by these classes --- the result is given in Theorem~\ref{thm:subring}.

The Chow and the cohomology rings of $\calX_{g-1}^2$ admit a natural automorphism which plays a key role in our computations. Let $s:\calX_{g-1}^2\to\calX_{g-1}^2$ denote the shift map defined by
$$
s(B,\beta,b)=(B,\beta+b,b),
$$
and let $s^*:A^*(\calX_{g-1}^2)\rightarrow A^*(\calX_{g-1}^2)$ denote the induced map on the Chow ring. The action of $s^*$ on the divisors $T_1$, $P$ and $T_2$ was computed in  \cite{grle},\cite{ergrhu2} to be
\begin{equation}\label{eq:shift}
s^*(T_2)=T_2;\quad s^*(P)=P+2T_2; \quad s^*(T_1)=T_1+P+T_2.
\end{equation}
(see also remark \ref{moonen} for an alternative viewpoint on the action).
The Chow ring of $A^*(\widetilde{Y})$ is generated over the Chow ring $A^*(\calX_{g-1}^2)$ by one class $\xi$ satisfying the relation $\xi^2=\xi P$ (see \cite{fultonintersection}). We think of $\xi$ as the class of the $0$-section, in which case $\xi-P$ is the class of the $\infty$-section, and the relation $\xi\cdot (\xi-P)=0$ expresses the fact that these sections do not intersect.

The action of the involution $j$ on $\widetilde Y$ is studied in detail in \cite[Sec.~4]{grhu2}, where it is described globally in coordinates. It is easy to see that $j$ interchanges the $0$- and $\infty$-sections of $\widetilde Y$, and thus its action on $A^*(\widetilde Y)$ interchanges $\xi$ and $\xi-P$, which implies in particular that $j^* P= -P$. From the explicit description of the action we then also see that $j^* T_i=T_i$ (since the theta divisors are symmetric).

\smallskip

We also consider several cycles on the entire partial compactification $\calX_g'$, and their pullbacks to $\widetilde{Y}$. The divisor $T$ extends to a universal polarization divisor $\T\in A^1(\calX_g')$. The boundary of $\calX_g'$ is an irreducible divisor, the class of which we denote $D$, therefore we have $\Pic_\QQ(\calX_g')=\QQ \lambda_1\oplus\QQ \Theta\oplus\QQ D$. Finally, we consider the class of the gluing locus $\Delta\in A^2(\calX_g')$.

Finally, we need to know how these cycles restrict to the boundary. Let $(\cdot)|_{\widetilde{Y}}:A^*(\calX_g')\rightarrow A^*(\widetilde{Y})$ denote the pullback map. Then by \cite{mumforddimag},\cite{grle}, and \cite{ergrhu2} we have
\begin{equation}
D|_{\widetilde{Y}}=-2T_2;\quad \Theta|_{\widetilde{Y}}=\xi+T_1-P/2.
\label{eq:howclassesrestrict}
\end{equation}
We compute the pullback of $\Delta$ to $\widetilde{Y}$ in Proposition~\ref{prop:deltaclass} (note that $\Delta|_\Delta$ was computed in \cite{ergrhu2}).

\section{Intersection theory on $\calX_{g-1}\times_{\calA_{g-1}}\calX_{g-1}$}\label{sec:semiabelic}
We prove our main result by restricting the formula (\ref{eqn:main}) to the boundary of the partial compactification of the universal family, and expressing all of the cycles involved in terms of the divisor classes $\xi$, $T_1$, $P$ and $T_2$ defined in the previous section. To compare products of cycles on the boundary, we first need to understand the subring of the Chow ring generated by these divisors.

In this section, we compute the subring of $A^*(\calX_{g-1}^2,\QQ)$ generated by the classes $T_1$, $P$ and $T_2$. We show there are no relations in codimension up to and including $g$, and that the ring is Gorenstein with socle in dimension $2g-2$. This calculation improves on the results of \cite{ergrhu2}, in particular on Theorem 7.1, which describes the pushforwards of products of $T_1$, $P$, and $T_2$ to the base $\calA_{g-1}$.
\begin{thm}\label{thm:subring}
Let $R$ denote the subring of $A^*(\calX_{g-1}^2,\QQ)$ generated by the classes $T_1$, $P$, and $T_2$, and let $R^k$ denote the subspace of $R$ spanned by monomials of degree $k$. Then
\begin{enumerate}
\item The ideal of relations in $R$ is generated by all the coefficients of the one basic relation
\begin{equation}\label{genrel}
(T_1+nP+n^2T_2)^g=0,\quad n\in \ZZ
\end{equation}
considered as a polynomial in $n$ (i.e.~by all the homogeneous in $n$ pieces of it).
In particular, there are no non-trivial relations between $T_1$, $P$ and $T_2$ in degree less than $g$.
\item $R$ is a Gorenstein ring with socle in codimension $2g-2$, in other words,
$$
R^{2g-2}\cong\QQ,\qquad \dim R^k=0\mbox{ for }k>2g-2,
$$
and for any $0\leq k\leq g-1$ the product map
$$
R^{g-1-k}\times R^{g-1+k}\rightarrow R^{2g-2}\cong \QQ
$$
is a perfect pairing (in particular $\dim R^{g-1-k}=\dim R^{g-1+k}$).  Moreover, the multiplication by $(T_1T_2)^{k}$ is an isomorphism from $R^{g-1-k}$ to $R^{g-1+k}$ .
\end{enumerate}
\end{thm}
\begin{rem}\label{moonen}
While we give an elementary direct computational proof of the theorem below, Moonen explained to us that this result can be deduced from a special case of his work \cite{moonen}, combined with the results of Thompson \cite{thompson}, stated in terms of representation theory.

Indeed, recall that given an ample divisor class $L$ on a projective variety $X$, multiplication by $L$ defines a degree $2$ operator $e$ on the cohomology ring $H^*(X,\QQ)$, which by the hard Lefschetz theorem extends to an $\mathfrak{sl}_2$-action on $H^*(X,\QQ)$, known as the {\it Lefschetz action}. It turns out that if $X$ is an abelian variety, such a Lefschetz action of $\mathfrak{sl}_2$ also exists on the Chow ring. This construction is implicit in K\"unnemann \cite{1993Kuennemann}, is explicitly described in Polishchuk's thesis \cite{1996Polishchuk}, and was recently described in greater generality by Beauville \cite{2010Beauville}.

In \cite{lolu}, Looijenga and Lunts considered a generalization of this Lefschetz action for the case when $X$ has many non-proportional ample classes. In this case, they construct a larger Lie algebra action on $H^*(X,\QQ)$ that includes an $\mathfrak{sl}_2$-action for each ample divisor class. For abelian varieties, Moonen \cite{moonen} recently constructed a corresponding Lie algebra action on the Chow ring. This Lie algebra includes $\mathfrak{sl}_2$ action (on the Chow) corresponding to any polarization class, and also contains the endomorphism algebra of an abelian variety.

For a very general ppav $X$, the N\'eron--Severi group $\op{NS}_{\QQ}(X\times X)$ is generated by $T_1$, $P$, and $T_2$, while the endomorphism algebra is $\mathfrak{gl}(2,\QQ)$, acting coordinatewise.
Then the results of Moonen imply that there exists an action of  $\mathfrak{sp}(4,\QQ)$ on $A^*(X\times X)$, lifting the action on cohomology constructed in \cite{lolu}, containing the endomorphism algebra $\mathfrak{gl}(2,\QQ)$ in the standard way, and such that the operators of multiplication by $T_1$, $P$ and $T_2$ lie in the Borel subalgebra.
Since the classes $T_1$, $P$, and $T_2$ lie in $H^{1,1}(X\times X)$, the action of $\mathfrak{sp}(4,\QQ)$ on them can be deduced, including in particular \eqref{eq:shift}. Furthermore, the ring $R$ generated by $T_1$, $P$, and $T_2$ (which, by the results of \cite{moonen} and \cite{lolu} is thus the same in Chow and cohomology) can be described completely as a representation of $\mathfrak{sp}(4,\QQ)$ --- this is the content of \cite[Theorem 3.4]{thompson}, which can be reinterpreted to show that in fact the ring $R$ forms an irreducible representation of $\mathfrak{sp}(4,\QQ)$ of weight $(g-1,g-1)$; this in particular implies that all relations in $R$ are given by (\ref{genrel}).
\end{rem}
We now prove the theorem, in an elementary direct way, by considering the action of the shift operator $s^*$ defined by (\ref{eq:shift}) on the ring $A^*(\calX_{g-1}^2)$ and by using Theorem~\ref{thm:0g}.
\begin{prop}\label{prop:rel}
For any integer $n$ relation (\ref{genrel}) holds in the Chow ring, i.e.~we have $(T_1+nP+n^2T_2)^g=0$ in $A^{g}(\calX_{g-1}^2,\QQ)$.
\end{prop}
\begin{proof}
The class $T_1$ is the pullback of the universal theta divisor $T$ on $\calX_{g-1}$. According to Theorem~\ref{thm:0g}, $T^{g-1}=(g-1)!Z_{g-1}\in A^{g-1}(\calX_{g-1})$. Multiplying both sides of this equality by $T$ and recalling that $T$ is trivial along the zero section, so that $TZ_{g-1}=0$, we see that $T^g$ is zero in the Chow ring. Pulling back this relation from $\calX_{g-1}$ to $\calX_{g-1}^2$ under $pr_1^*$ we get that $T_1^g=0$ in $A^{g}(\calX_{g-1}^2)$.

We now apply the shift operator to this relation. A direct calculation using (\ref{eq:shift}) shows that $(s^*)^n(T_1)=T_1+nP+n^2T_2$. The shift operator $s^*$ is an automorphism of $A^*(\calX_{g-1}^2)$, so the result follows.
\end{proof}
\begin{proof}[Proof of Theorem~\ref{thm:subring}]
The proof is direct and computational. We prove the statements of the theorem in the following order. First, we introduce a second grading that distinguishes the divisors $T_1$, $P$ and $T_2$. We then prove the vanishing of $R^k$ for $k>2g-2$ and show that multiplication by $(T_1T_2)^k$ is a surjective map from $R^{g-1-k}$ to $R^{g-1+k}$. Then we use a pushforward calculation to prove that $R^{2g-2}\cong \QQ$. We then show that the relations (\ref{genrel}) generate the ideal of relations, and that multiplication by $(T_1T_2)^k$ is also injective. Finally, we prove the perfect pairing statement.

\subsection{Second grading}
We define a second grading $d$ on $R$ by setting
$$
d(T_1^aP^bT_2^c):=a-c.
$$
This grading is motivated by the fact that all the summands of the $n^{g-k}$ term of relation (\ref{genrel}) have degree $k$ in this grading.
We now
consider the decomposition of $R$ with respect to $d$ and the usual degree:
$$
R=\displaystyle\bigoplus_{k=0}^{\infty}R^k=\displaystyle\bigoplus_{k=0}^{\infty}\displaystyle\bigoplus_{l=-k}^k R^k_l,\quad R^k_l=\left\{L\in R|\deg(L)=k,\,d(L)=l\right\}.
$$

We consider relation (\ref{genrel}) as a polynomial of degree $2g$ in a variable $n$. The coefficients of this polynomial give $2g+1$ relations in $R^g$, one in each $R^g_l$:
\begin{equation} \label{eq:positivegzero}
\displaystyle\sum_{m=0}^{\left\lfloor(g-l)/2\right\rfloor}T_1^{l+m}P^{g-l-2m}T_2^m
\frac{g!}{(l+m)!(g-l-2m)!m!}=0\in R^g_{l},\,\,0\leq l\leq g,
\end{equation}
\begin{equation}
\displaystyle\sum_{m=0}^{\left\lfloor(g+l)/2\right\rfloor}T_1^m P^{g+l-2m}T_2^{m-l}\frac{g!}{m!(g+l-2m)!(m-l)!}=0\in R^g_{l}\mbox,\,-g\leq l< 0.
\label{eq:negativegzero}
\end{equation}

\subsection{Vanishing for $R_l^k$} We first show that $R_l^k$ vanishes for $|l|\geq 2g-k-1$. Indeed, relation (\ref{eq:positivegzero}) for $l=g-1$ implies that $T_1^{g-2}P=0$. Multiplying (\ref{eq:positivegzero}) by powers of $T_1$, we obtain by induction on $p$ that
$$
T_1^{g-1-p}P^{2p+1}=0\mbox{ for }0\leq p\leq g-1.
$$
Now suppose that $T_1^aP^bT_2^c$ is an element of $R_k^l$ for $l\geq 2g-k-1$. Comparing the inequalities, we see that either $a\geq g$, or $b\geq 2g$, or $T_1^aP^bT_2^c$ is a multiple of $T_1^{g-1-l}P^{2l+1}$ for $l=\left\lfloor (b-1)/2\right\rfloor$. In all cases this term is zero, hence $R^k_l=0$ for $l\geq 2g-k-1$. A similar proof shows that $R^k_l=0$ for $l\leq -2g+k+1$.

It now follows that $R^k=0$ for $k>2g-2$, because for any integer $l$ we have that $|l|\geq  0\geq 2g-k-1$, and hence all graded components $R_l^k=0$ vanish.

\subsection{Surjectivity of multiplication by $(T_1T_2)^k$} We next show that every element of $R^{g-1+k}$ can be written as a multiple of $(T_1T_2)^k$. We have already seen above that $R^{g-1+k}_l=0$ for $|l|>g-1-k$. Now suppose $l\leq g-1-k$, and assume without loss of generality that $l\geq 0$. Let $T_1^aP^bT_2^c\in R^{g-1+k}_l$, so that $a+b+c=g-1+k$ and $a-c=l$. If $c\geq k$ then $a=c+l\geq k$, so $T_1^aP^bT_2^c$ is a multiple of $(T_1T_2)^k$. If $c<k$, then $a+b\geq g$, and we can use relation (\ref{eq:positivegzero}) to express $T_1^aP^b$ as a multiple of $T_1T_2$. Repeating this procedure if necessary, we can raise the exponent of $T_2$ to $k$ and write $T_1^aP^bT_2^c$ as a multiple of $(T_1T_2)^k$, which proves the surjectivity of multiplication by $(T_1T_2)^k$.

\subsection{$R^{2g-2}$ has dimension one} We now show that $R^{2g-2}\cong \QQ$. We have already seen that $R^{2g-2}_l=0$ for $l\neq 0$, and that every term in $R^{2g-2}_0$ is a multiple $(T_1T_2)^{g-1}$. Hence $R^{2g-2}$ is at most one-dimensional. The pushforwards of the classes in $R^{2g-2}_0$ to $A^*(\calA_{g-1})$ were computed in \cite[Theorem 7.1]{ergrhu2}:
\begin{equation}\label{eq:pushforward}
h_*(T_1^{g-1-a}P^{2a}T_2^{g-1-a})=(-1)^a\frac{(g-1)!(2a)!(g-1-a)!}{a!}[\calA_{g-1}].
\end{equation}
Therefore, all of the classes $T_1^{g-1-a}P^{2a}T_2^{g-1-a}$ are non-zero, and so $R^{2g-2}$ has dimension one.

\subsection{Linear independence} We have shown that $R^k$ is spanned as a vector space by monomials that are multiples of $(T_1T_2)^{k-g+1}$. We now show that these monomials are linearly independent, by induction on $k$ from $k=2g-2$ down to $k=g-1$. This will prove both that multiplication by $(T_1T_2)^{k-g+1}$ is an isomorphism from $R^{2g-2-k}$ to $R^k$ for $k\geq g$, and that there are no relations in degree less than $g$.

The base case, namely that $T_1^{g-1}T_2^{g-1}$ is non-zero, was established above. Now suppose that we have a linear relation in $R^k$ for some $g-1\leq k<2g-2$. We split it up according to the $d$-grading:
\begin{equation}
(T_1T_2)^{k-g+1}\cdot (X_++X_0+X_-)=0,
\label{eq:linearindependence}
\end{equation}
where $X_+$, $X_0$ and $X_-$ denote the sum of monomials with positive, zero and negative $d$-grading, respectively. We multiply this relation by $T_2$. We know that $X_+=T_1Y_1$, where $Y_1$ only consists of monomials with non-negative $d$-grading. Using (\ref{eq:negativegzero}) we can express $(T_1T_2)^{k-g+1}(X_0+X_-)=(T_1T_2)^{k-g+2}Y_2$, where $Y_2$ only consists of monomials with negative $d$-grading. By induction, monomials that are multiples of $(T_1T_2)^{k-g+2}$ are linearly independent in $R^{k+1}$, hence (\ref{eq:linearindependence}) implies that $Y_1+Y_2=0$. Since these two terms have distinct $d$-grades, it follows that $Y_1=0$, and hence $X_+=0$.  Similarly, multiplying by $T_1$ shows that $X_-=0$.

It remains to show that all $X_0=0$, i.e. that there are no non-trivial relations in $R^k_0$. The reasoning is similar. If $k$ is odd, we multiply relation (\ref{eq:linearindependence}) by $P$ to obtain a relation in $R^{k+1}_0$. Using relation (\ref{eq:positivegzero}) to express each term as a multiple of $(T_1T_2)^{k-g+2}$ and induction, we see that $X_0=0$. If $k$ is even, then $R^k_0$ has dimension one greater than $R^{k+1}_0$, and it is also necessary to multiply relation (\ref{eq:linearindependence}) by $T_1$ (or $T_2$) and use the induction hypothesis for $R^{k+1}_1$.

We have shown monomials that are multiples of $(T_1T_2)^{k-g+1}$ form a basis for $R^k$ for $g-1\leq k\leq 2g-2$. This proves that the multiplication by $(T_1T_2)^{k-g+1}$ map from $R^{2g-2-k}$ to $R^k$ is an isomorphism, and that there are no other relations in the ring $R$. In particular, we have shown that there are no relations in $R$ in degree less than $g$.

\subsection{Perfect pairing}

Finally, we need to show that the product map defines a perfect pairing
$$
R^{k}\times R^{2g-2-k}\rightarrow R^{2g-2}\simeq \QQ.
$$
We first split up by $d$-grading. We first note that $R^k_{l}=R^{k-l}_0\cdot T_1^l$ and $R^{2g-2-k}_{-l}=R^{2g-2-k-l}_0\cdot T_2^l$, and that $R^k_0=R^{k-1}_0\cdot P$ for $k$ odd. Hence, it is sufficient to prove that $R^k_0\times R^{2g-2-k}_0\rightarrow R^{2g-2}_0$ is a perfect pairing for $1\leq k\leq g-1$.

We prove this by induction on $k$. For $k=1$, suppose that $X=aT_1^{g-2}T_2^{g-2}+bT_1^{g-3}P^2T_2^{g-3}$ in $R^{2g-4}$ pairs to zero with $R_0^2$. Multiplying $X$ by $T_1T_2$ and $P^2$ and taking the pushforward to the base using (\ref{eq:pushforward}), we see that $a=b=0$.

Similarly, suppose that the element $X\in R^{2g-2-2k}_0$ pairs to zero with $R^{2k}_0$. Then the elements $X T_1T_2$ and $XP^2$ in $R^{2g-2k}_0$ kill $R^{2k-2}_0$, so by induction they are zero. Using relations (\ref{eq:positivegzero}) and the linear independence of multiples of $(T_1T_2)^{g+1-2k}$ in $R^{2g-2k}_0$ we see that $X=0$, which proves the perfect pairing statement.
\end{proof}

\section{Shift-invariant classes}\label{sec:shift}

Our goal is to compute the restriction of the zero section of the universal semiabelian variety to the boundary $Y$ in terms of products of pullbacks of geometric cycles defined on the whole family $\calX_g'$. The boundary $Y$ is not a normal stack, so the Chow group $A^*(Y)$ does not have an intersection product. To avoid this difficulty we instead work in the Chow ring $A^*(\widetilde{Y})$, where $\widetilde{Y}$ is a 2-to-1 cover of the normalization of the boundary. For this reason, we need to determine which cycles in $A^*(\widetilde{Y})$ are pullbacks of cycles from $A^*(Y)$, and in particular pullbacks of intersections of cycles on $\calX_g'$ with $Y$. We denote by $(\cdot)|_{\widetilde{Y}}:A^*(\calX_g')\to A^*(\widetilde{Y})$ the composition of the restriction to $Y$ with the pullback to $\widetilde Y$.

The Chow ring $A^*(\widetilde{Y})$ is generated over the Chow ring $A^*(\calX_{g-1}^2)$ by the class $\xi$ of the zero section satisfying the relation $\xi^2-\xi P=0$. In the previous section we determined the subring $R$ of $A^*(\calX_{g-1}^2)$ generated by the classes $T_1$, $P$ and $T_2$. In this section, we describe the classes in $\widetilde{R}\subset A^*(\widetilde{Y})$ that are pullbacks of classes from $Y$, where $\widetilde{R}=R[\xi]/(\xi^2-\xi P)$ denotes the subring of $A^*(\widetilde{Y})$ generated by $T_1$, $P$, $T_2$ and $\xi$. By abuse of notation, we will also use $T_1$, $P$ and $T_2$ to denote the pullbacks of these classes to $A^1(\widetilde{Y})$.

The boundary $Y$ is the quotient by the involution $j$ of the $\PP^1$-bundle $\widetilde{Y}$ over $\calX_{g-1}^2$, with the zero section $\Delta_0$ glued to the infinity section $\Delta_{\infty}$ by a shift, resulting in the locus $\Delta\subset Y$. The two sections $\Delta_0$ and $\Delta_{\infty}$ define pullback maps $(\cdot)|_0$ and $(\cdot)|_{\infty}$ from $A^*(\widetilde{Y})$ to $A^*(\calX_{g-1}^2)$. By definition $\xi$ is the class of the zero section $\Delta_0$, hence
$$
\xi\cdot \Delta_0=\xi^2=\xi\cdot P=P\cdot \Delta_0.
$$
Therefore, the map $(\cdot)|_0:A^*(\widetilde{Y})\rightarrow A^*(\calX_{g-1}^2)$ consists in setting $\xi=P$. Similarly, the class of the infinity section $\Delta_{\infty}$ is $\xi-P$, hence
$$
\xi\cdot \Delta_{\infty}=\xi(\xi-P)=0,
$$
and the map $(\cdot)|_{\infty}:A^*(\widetilde{Y})\rightarrow A^*(\calX_{g-1}^2)$ consists in setting $\xi=0$.

Given a subvariety $V\subset Y$, the preimage of $V\cap\Delta$ in $\widetilde Y$ consists of two connected components, namely the preimages of $V$ in $\widetilde{Y}$ intersected with $\Delta_0$ and $\Delta_\infty$. Therefore, a class $X\in A^*(\widetilde{Y})$ is the pullback of a class from $A^*(Y)$ only if it is {\it shift-invariant}, in other words only if
\begin{equation}
s^*(X|_{\Delta_\infty})=X|_{\Delta_0},
\label{eq:shift-invariant1}
\end{equation}
where the above equality is in $A^*(\calX_{g-1}^2)$.

We also recall from \cite[Sec.~4]{grhu2} and from the discussion in Section 2 that the action of the involution $j$ on the semiabelic fibers of the universal family induces the following action on the Picard group:
$$
j^*\xi=\xi-P,\quad j^*P=-P,\quad j^* T_1=T_1,\quad j^* T_2=T_2.
$$

We now describe the shift-invariant and $j$-invariant classes.
\begin{prop}\label{prop:pullback}
Let $\widetilde{R}$ denote the ring $\QQ[\xi,T_1,P,T_2]/I$, where $I$ is the ideal generated by $\xi^2-\xi P$ and relations (\ref{genrel}). Let $j:\widetilde{R}\to\widetilde{R}$ denote the automorphism defined on the generators by
$$
j(\xi)=\xi-P,\quad j(P)=-P,\quad j(T_1)=T_1,\quad j(T_2)=T_2,
$$
and let $s$ be the shift operator defined on the subring generated by $T_1$, $P$ and $T_2$ as follows:
$$
s(T_1)=T_1+P+T_2,\quad s(P)=P+2T_2,\quad s(T_2)=T_2.
$$
Then the subset of elements $X\in\widetilde{R}$ that are $j$-invariant and that are shift-invariant:
$$
j(X)=X,\quad s(X(0,T_1,P,T_2))=X(P,T_1,P,T_2)
$$
is the subring generated by the classes $\Theta:=\xi+T_1-P/2$, $D:=-2T_2$, and $-4\xi T_2-P^2+2P T_2$.
\end{prop}
\begin{rem}
The notation $\Theta$ and $D$ is due to the fact that these are in fact the restrictions of the corresponding classes on $\calX_g'$, according to (\ref{eq:howclassesrestrict}). The next proposition shows that the third class in fact the restriction of $\Delta$.
\end{rem}
\begin{proof}
We first consider the automorphism $j$ on the free polynomial ring $\QQ[\xi,T_1,P,T_2]$. It is clear that $j$ is an involution, and that the $j$-invariant subring is generated by $\xi-P/2$, $P^2$, $T_1$ and $T_2$:
$$
\QQ[\xi,T_1,P,T_2]^{j}=\QQ[\xi-P/2,P^2,T_1,T_2].
$$
Let $r:\QQ[\xi,T_1,P,T_2]\to \widetilde{R}$ denote the projection map. First we note that $j$ preserves the ideal $I$, hence $j$ in fact descends to an involution of $\widetilde{R}$.

Suppose that $X\in \widetilde{R}$ satisfies $j(X)=X$. If $X=r(Y)$, then setting $Z=(Y+j(Y))/2$ we see that $X=r(Z)$ and $j(Z)=Z$. In other words, every $j$-invariant element in $\widetilde{R}$ is the image of a $j$-invariant element in $\QQ[\xi,T_1,P,T_2]$. Since $(\xi-P/2)^2=P^2/4$ in $\widetilde{R}$, we see that $\widetilde{R}^{j}$ is generated by $\xi-P/2$, $T_1$ and $T_2$.

The shift operator $s$ does not extend to the entire ring $\widetilde{R}$, so we cannot compute the subring of shift-invariant classes in the same way, as an invariant subring of the action of a finite group. However, we make the following observation. Let $S\subset\QQ[\xi,T_1,P,T_2]$ denote the subring generated by the classes $\Theta=\xi+T_1-P/2$, $\mu=\xi-P/2$ and $T_2$. We have shown above that $r(S)=\widetilde{R}^{j}$. It turns out that the ring $S$ admits an involution such that the subring of fixed elements is precisely the subring of shift-invariant classes.

Indeed, define an automorphism $\sigma$ of $S$ on the generators as follows:
$$
\sigma(\Theta)=\Theta,\quad \sigma(\mu)=-\mu-T_2,\quad \sigma(T_2)=T_2.
$$
The automorphism $\sigma$ preserves the ideal $S\cap I$ and it is an involution,  therefore $\sigma$ descends to an involution on $\widetilde{R}^{j}$. Moreover,  an element $X\in\widetilde{R}^{j}$ satisfies the gluing condition if and only if it is $\sigma$-invariant. Using the same reasoning as above, we see that subset of elements of $\widetilde{R}^{j}$ satisfying the gluing condition is the image under $r$ of the invariant subring $S^{\sigma}$. The invariant subring $S^{\sigma}$ is generated by $\Theta$, $D$, and the class $\mu\cdot \sigma(\mu)=-\mu(\mu+T_2)=-(\xi T_2+P^2/4-PT_2/2)$, which proves the theorem.
\end{proof}

We now give an interpretation for the third invariant class appearing in Proposition \ref{prop:pullback}:
\begin{prop}\label{prop:deltaclass}
The pullback of $\Delta$, considered as a class in $A^2(\calX_g')$, to $\widetilde{Y}$ is equal to
$$
\Delta|_{\widetilde{Y}}=-4\xi T_2-P^2+2P T_2=(2\xi -P)(-2\xi+P-2T_2)\in A^2(\widetilde Y).
$$
\end{prop}
\begin{proof}
The proof of this formula is a slight extension of the ideas of the proof of \cite[Prop.~4.3]{ergrhu2}, where it is shown that $\Delta|_\Delta=P(-P-2T_2)$. We note that the formula above restricts to this expression when we set $\xi=P$ (which we think of as restricting to the $0$-section), while for $\xi=0$ (the $\infty$-section) the above formula restricts to $-P^2+2PT_2$, which is obtained from $P(-P-2T_2)$ by sending $P$ to $-P$, which we know to be the action of the involution $j$ on $\Pic(\widetilde{Y})$.

To prove the proposition we interpret the class $\Delta$ geometrically. Indeed, recall from \cite{ergrhu2} that $\Delta$ is the locus where $Y$ is not normal, and thus in a small neighborhood of itself $\Delta$ is the intersection of the two local irreducible components of the locus $Y\subset\calX_g'$. Therefore the class of $\Delta$ is a product of divisors, and so lies in the ring generated by $\xi$, $T_1$, $P$ and $T_2$. The class of $\Delta$ also satisfies the conditions of Proposition~\ref{prop:pullback}, hence it is a linear combination of $\Theta^2$, $\Theta D$, $D^2$ and $-4\xi T_2-P^2+2PT_2$. Finally, $\Delta$ restricts to $-P^2+2PT_2$ when we set $\xi=P$, and it is easy to see that $-4\xi T_2-P^2+2PT_2$ is the only class that satisfies this condition.


\end{proof}

\begin{proof}[Proof of Theorem~\ref{thm:expressible}] The result now immediately follows from Proposition~\ref{prop:pullback} and Proposition~\ref{prop:deltaclass}.
\end{proof}

\begin{rem}
In the next section, we show that the restriction of the zero section to the boundary is a polynomial in $\xi$ and $T_1$, and therefore can be expressed as a polynomial in $\Theta$, $D$ and $\Delta$. For now, we note two curious facts.

First, we note that the class $Q:=\Delta-2\Theta D=4T_1T_2-P^2$ does not contain $\xi$, and is therefore in the image of $A^2(\Delta)$ in $A^2(\widetilde Y)$, and the expression for the zero section in terms of the class $Q$ is much simpler than in terms of $\Delta$ (see Theorem~\ref{thm:main}).

Second, we note that one can show that the subring of $R^*$ invariant under gluing (i.e.~under the involution $\sigma$) is generated by $\Theta$, $D$, and $\Delta$, together with one additional class, $\xi(6PT_2+12T_2^2)+P^3-4PT_2^2$, that satisfies a quadratic relation in $\Theta$, $D$, and $\Delta$. We do not know if this class has any geometric meaning.
\end{rem}

\section{Class of the partial boundary of the zero section}\label{sec:main}

We now prove Theorem~\ref{thm:main}, obtaining an explicit expression for the class of the locus of the closure of the zero section in the partial compactification.

Our goal is to extend Theorem~\ref{thm:0g} to the partial compactification. Denote by $z_g':\calA_g'\to\calX_g'$  the closure of the zero section in the partial compactification of the universal family, and denote, as above, by $\T\subset\calX_g'$ the closure of the theta divisor and its class. In \cite{vdgeerchowa3} van der Geer computes the Chow rings of $\overline{\calA}_3$ and $\calX_2'$, and in particular shows that $Z_g'\ne [T^g]/g!$ in $A^2(\calX_2')$. It is easy to deduce that such an equality does not hold in any higher genus either. We now compute the difference.

\smallskip

We describe the locus $z_g'$ explicitly using our description of the geometry of $\calX_g'$, as the universal space of the universal Poincar\'e bundle over the universal fiberwise product $\calX_{g-1}^2=\calX_{g-1}\times_{\calA_{g-1}}\calX_{g-1}$. Indeed, the semiabelic variety of torus rank one is no longer a group, but is acted upon by the semiabelian variety (the $\CC^*$-bundle over the same base $B$), which is a group. The zero for the group law of the semiabelian variety is the point $1\in \CC^*$ lying in the fiber over the zero in the base abelian variety $B$. The zero of the semiabelic variety becomes one of the limits of two-torsion points on it (as described in detail in \cite{grhu2})  --- which one, it does not matter for us, as their classes are all equivalent modulo torsion, and we are working in the Chow ring with rational coefficients. Thus the restriction of $z_g'$ to the boundary $\calX_{g-1}$ of $\calA_g'$ is the map that associates to $(B,b)\in\calX_{g-1}$ the point $(B,0,b,1)\in Y=\partial\calX_g'$. This is of course a section of the universal Poincar\'e bundle restricted to the locus $\lbrace(B,0,b)\rbrace$, and thus its class $\partial Z_g':={Z_g'}|_Y$ is equal to $\xi$ times the class of the locus $\lbrace( B,0,b)\rbrace\subset\calX_{g-1}^2$. However, this class is just the class of the zero section $z_{g-1}:\calA_{g-1}\to\calX_{g-1}$, pulled back to $\calX_{g-1}^2$ under $pr_1$. By Theorem~\ref{thm:0g} discussed above, this is the pullback of the class $T^{g-1}/(g-1)!$ under the projection map $pr_1$, i.e.~the class $T_1^{g-1}/(g-1)!$ in our notation. Therefore, we have proved the following result:

\begin{prop}\label{prop:xiT} The class of the restriction to $\widetilde{Y}$ of the closure of the zero section $Z_g'$ is equal to
$$
\partial Z_g'=\frac{\xi T_1^{g-1}}{(g-1)!}\in A^g(\widetilde Y).
$$
\end{prop}

Notice that there is an ambiguity here: we could have as well deduced the same formula with $\xi$ replaced by $\xi+P$, by arguing that the 1-section of the $\PP^1$-bundle is also a section over the $B$ that is the $\infty$-section, instead of the $0$-section, with the corresponding shift. This is consistent, as $T_1^{g-1}P=0\in A^g(\calX_{g-1}^2)$ by Proposition~\ref{prop:rel}. Of course the zero section, being defined geometrically on $Y$, pulls back to a shift-invariant class on the normalization $\widetilde Y$ of $Y$, and Theorem~\ref{thm:expressible} applies to show that $\partial Z_g'$ is a polynomial in the classes $\T$, $D$, and $\Delta$. It remains to compute the coefficients, proving our main result.

\begin{proof}[Proof of the main theorem~\ref{thm:main}]
We first note that the class $\partial Z_g'=\frac{\xi T_1^{g-1}}{(g-1)!}$ satisfies the conditions of Proposition~\ref{prop:pullback} (it is shift-invariant since $T_1^{g-1}P=0$). Therefore, it can be written as a polynomial in $\Theta$, $D$ and $\Delta$. It turns out that the formula for the zero section is simpler in terms of the alternative classes $\Theta-D/8$, $D$, and $\Delta-2\Theta D$.

These three classes also generate the subring of shift-invariant polynomials, therefore there exists a formula
\begin{equation}\label{eq:mainproof}
\frac{\xi T_1^{g-1}}{(g-1)!}= \displaystyle\sum_{a+b+2c=g}\alpha_{a,b,c}(\Theta-D/8)^a D^b(\Delta-2\Theta D)^c,
\end{equation}
where the classes $\T$, $D$ and $\Delta$ are given in terms of $\xi$, $T_1$, $P$ and $T_2$ by
$$
\T=\xi+T_1-\frac{P}{2},\quad D=-2T_2,\quad \Delta=-4\xi T_2-P^2+2PT_2.
$$
We first find the coefficients $\alpha_{a,0,c}$ not involving $D$.

In the main equation (\ref{eq:mainproof}), set $T_2=0$, obtaining
$$
\frac{\xi T_1^{g-1}}{(g-1)!}=\displaystyle\sum_{a+2c=g}\alpha_{a,0,c}\left(\xi+T_1-\frac{P}{2}\right)^a(-P^2)^c.
$$
For an arbitrary integer $n$ we now formally set $T_1=\left(n+\frac{1}{2}\right)P$. Using $\xi^2=\xi P$ we then get
$$
(\xi+nP)^a=n^aP^a+\displaystyle\sum_{i=1}^an^{a-i}\left(\begin{array}{c} a \\ i\end{array}\right)\xi P^{a-1}=n^aP^a+[(n+1)^a-n^a]\xi P^{a-1}.
$$
Therefore, equating the coefficients in front of $\xi P^{g-1}$ on both sides gives
$$
\frac{\left(n+\frac{1}{2}\right)^{g-1}}{(g-1)!}=\displaystyle\sum_{a+2c=g}\alpha_{a,0,c}[(n+1)^a-n^a]
(-1)^c.
$$
We now sum this equality from $n=1$ to $n=N-1$, where $N$ is another integer. The left hand side can be expressed in terms of Bernoulli numbers:
$$
\displaystyle\sum_{n=1}^{N-1}\left(n+\frac{1}{2}\right)^{g-1}=\frac{1}{2^{g-1}}\left[\displaystyle\sum_{k=1}^{2N}k^{g-1}-\displaystyle\sum_{l=1}^N(2l)^{g-1}-1\right]=
$$
$$
=\displaystyle\sum_{m=0}^{g-1}\frac{N^{g-m}B_m}{m!(g-m)!}(2^{1-m}-1)-\frac{1}{2^{g-1}}.
$$
Comparing this with the right hand side and equating coefficients of the powers of $N$ yields
$$
\alpha_{a,0,c}=\frac{(-1)^c}{a!(2c)!}(2^{1-2c}-1)B_{2c},
$$
as claimed by the theorem.

For the coefficients $\alpha_{a,b,c}$ with $b>0$, we do not know an elegant derivation as above. Instead, we show that the remaining coefficients satisfy a triangular system of equations in terms of the coefficients $\alpha_{a,0,c}$, and solve this system directly using Maple.
We consider the main equation (\ref{eq:mainproof}), and set $\xi=0$:
$$
\displaystyle\sum_{a+b+2c=g}\alpha_{a,b,c}\left(T_1-\frac{P}{2}+\frac{T_2}{4}\right)^a (-2T_2)^b(4T_1T_2-P^2)^c=0.
$$
Now formally apply the square root of the shift operator (\ref{eq:shift})
$$
(s^*)^{1/2}(T_1)=T_1+\frac{P}{2}+\frac{T_2}{4},\quad (s^*)^{1/2}(P)=P+T_2,\quad (s^*)^{1/2}(T_2)=T_2,
$$
to this equation. We get that
$$
\displaystyle\sum_{a+b+2c=g}\alpha_{a,b,c}T_1^a (-2T_2)^b(4T_1T_2-P^2)^c=0.
$$
This is a relation in the ring $R^*$, in other words this equation is a linear combination of relations (\ref{eq:positivegzero})-(\ref{eq:negativegzero}). These relations are homogeneous with respect to the grading $d$, as well as the usual grading, so the $d$-homogeneous parts of the above equation vanish separately. The possible values of the grading $d$ are $g-2h$, where $h=0,\ldots,g$, so the above equation splits into the following system:
$$
\displaystyle\sum_{c=0}^{\min(h,g-h)} \alpha_{g-h-c,h-c,c}T_1^{g-h-c}(-2T_2)^{h-c}(4T_1T_2-P^2)^c=0,\quad h=0,\ldots,g.
$$
First, assume that $g-h\geq h$. Expanding $(4T_1T_2-P^2)^c$ and changing the order of summation, we can write the above as
$$
\displaystyle\sum_{l=0}^{h} T_1^{g-h-l}P^{2l}T_2^{h-l}\frac{(-1)^{h+l}2^{h-2l}}{l!}\displaystyle\sum_{c=l}^h
\frac{c!}{(c-l)!}(-1)^c2^c\alpha_{g-h-c,h-c,c}=0.
$$
This equation is satisfied if and only if the left hand side is a multiple of the corresponding relation (\ref{eq:positivegzero}). This gives us a triangular system of equations on the coefficients $\alpha_{g-h-c,h-c,c}$, and we have already determined the coefficient $\alpha_{g-2c,0,c}$ above, so the remaining coefficients are determined uniquely by this system.

Therefore, to prove Theorem~\ref{thm:main} it is sufficient to substitute the coefficients (\ref{eqn:beta}) into the formula above and check that we get relation (\ref{eq:positivegzero}). Substituting and dividing out by a common multiple, we get
$$
\displaystyle\sum_{l=0}^{h} T_1^{g-h-l}P^{2l}T_2^{h-l}\frac{(-1)^l 2^{-2l}}{l!}\displaystyle\sum_{c=l}^h
\frac{(-1)^c 2^{2c}(2g-2c)!}{(g-c)!(c-l)!(g-h-c)!(h-c)!}=0.
$$
Using Maple, we evaluate the inside sum as
$$
\displaystyle\sum_{c=l}^h
\frac{(-1)^c 2^{2c}(2g-2c)!}{(g-c)!(c-l)!(g-h-c)!(h-c)!}=C_{g,h}
\frac{(-1)^l2^{2l}l!}{(g-l-h)!(h-l)!(2l)!},
$$
where $C_{g,h}$ is a coefficient depending on $g$ and $h$. Plugging this into the equation above, we see that it is indeed a multiple of (\ref{eq:positivegzero}), hence it is satisfied. This completes the calculation of the coefficients $\alpha_{g-h-c,h-c,c}$ for $g-h\geq h$, and the calculation of the other coefficients is virtually identical.

Finally, the coefficients $\eta_{a,b,c}$ are obtained by expanding formula $(\ref{eqn:main})$ and using Maple to simplify.
\end{proof}
\begin{rem}\label{rem:goingtoAg}
Given the explicit formula we obtain for the extension of the zero section to the partial compactification, it is natural to ask whether a formula for the extension to the next boundary stratum (over the locus of torus rank two semiabelic varieties) may be possible. This locus of semiabelic varieties of torus rank two is the same for perfect cone, second Voronoi, and central cone (Igusa) toroidal compactifications --- since all these compactifications coincide in genus 2, and restrict inductively to products. In principle it should be possible to describe explicitly the geometry of the universal family of semiabelic varieties of torus rank two (which can now be of two types, depending on whether the normalization is a $\PP^1\times\PP^1$ bundle, or two copies of a $\PP^2$ bundle). This computation would be very involved technically, but could shed further light on the class of the closure of the zero section in $\Perf$, which would be instrumental in trying to inductively describe its cohomology. We note also that the fact that torus rank up to two strata of a toroidal compactification of $\calA_g$ are closely related to the partial compactification of the universal family does not seem to extend deeper, as even the existence of a universal family over $\Perf$ is not known globally.
\end{rem}

\section{Extension of the double ramification cycle}\label{sec:eliashberg}
In this section we extend Hain's formula for the double ramification cycle from $\calM_{g,n}^{ct}$ to $\overline{\calM}_{g,n}^o$, the locus of curves having at most one non-separating node. We recall the setup. Fix a list of integers $\ud=(d_1,\ldots,d_n)$ such that $\sum d_i=0$. The {\it double ramification cycle} $R_\ud\subset\calM_{g,n}$ is defined as the locus of curves $(X,p_1,\ldots,p_n)$ such that the sum $\sum d_i p_i$ is a principal divisor on $X$. The locus $R_\ud$ is very natural from the point of view of Hurwitz theory. This locus, or related loci (see eg. \cite{muller}) also occurs naturally in various enumerative problems, and is also studied in Gromov--Witten theory, see \cite{FSZ} for more references and a discussion.

We approach this locus in the following way. Given $\ud=(d_1,\ldots,d_n)$, denote $s_{\ud}:\calM_{g,n}\rightarrow \calX_g$ the Abel--Jacobi map that associates to the marked curve $(X,p_1,\ldots,p_n)$ the line bundle $\mathcal{O}_X(\sum d_i p_i)$ on $\op{Jac}(X)$ (where in this section we will always think of $\Jac(X)$ as $\Pic^0(X)$). Then $R_\ud$ is the locus in $\calM_{g,n}$ where $s_{\ud}(X,p_1,\ldots,p_n)=0\in\op{Jac}(X)$, so we can compute $R_\ud$ by pulling back the zero section of the universal abelian variety under the map $s_\ud$.

To compute the closure of the double ramification cycle in $\overline{\calM}_{g,n}^o$, we need to understand how the map $s_\ud$ extends to $\overline{\calM}_{g,n}^o$. We first recall the extension of the Abel--Jacobi map to curves of compact type, as recalled in \cite{grza1}, and described in references therein.

Let $(X,p_1,\ldots,p_n)$ be a smooth marked curve of genus $g$. Fix a basis $A_i, B_i$ of $H_1(X,\ZZ)$, and let $\omega_i$ be a basis for $H^0(X,\Omega)$ dual to the cycles $A_i$. Identifying $\op{Jac}(X)$ with $H^0(X,\Omega^1)^{\vee}/H_1(X,\mathbb{Z})$, the Abel--Jacobi map $s_{\ud}$ is given by
\begin{equation}
s_{\ud}\left(X, p_1,\ldots,p_n\right)=\displaystyle\sum_{i=1}^n d_i \left(\displaystyle\int_q^{p_i} \omega_1,\ldots,
\displaystyle\int_q^{p_i}\omega_g\right)\in \op{Jac}(X),
\label{eq:AJ}
\end{equation}
where $q\in X$ is an arbitrary base point. We obtain a description of $s_{\ud}$ on singular curves by considering degenerations of the above formula.

First, suppose that $(X_t,p_1,\ldots,p_n)$ is a family of smooth marked curve degenerating as $t\rightarrow 0$ to a curve of compact type $X=X_0$ having irreducible components $X'$ and $X''$ of genera $h$ and $g-h$, respectively, joined at the points $q'\in X'$ and $q''\in X''$ to form a node. Assume without loss of generality that the points $p_i$ for $i=1,\ldots,k$ and the point $q$, as well as the cycles $A_j$ and $B_j$ for $j=1,\ldots,h$ end up on $X'$, and the remaining points and cycles on $X''$. Denote $e=-(d_1+\cdots+d_k)$. The limit of $\omega_j$ is a normalized $1$-form on one of the components ($X'$ for $j\leq h$ and $X''$ otherwise) and zero on the other. Hence,  for $j=1,\ldots,h$ the limit of $\int_q^{p_i}\omega_j$ is $\int_q^{p_i}\omega_j$ for $i=1,\ldots,k$ and $\int_p^{q'}\omega_j$ for $i=k+1,\ldots,n$, and similarly for $j=h+1,\ldots,g$. We therefore see that
\begin{equation}
s_{\ud}(X,p_1,\ldots,p_n)=(s_{\ud'}(X',p_1,\ldots,p_k,q'),s_{\ud''}(X'',p_{k+1},\ldots,p_n,q'))\in \op{Jac}(X),
\label{eq:AJcompact}
\end{equation}
where we identify $\op{Jac}(X)=\op{Jac}(X')\times \op{Jac}(X'')$, and denote $\ud'=(d_1,\ldots,d_k,e)$ and $\ud''=(d_{k+1},\ldots,d_n,-e)$.

The Abel--Jacobi map for an arbitrary curve of compact type is obtained inductively using the above procedure, by using a sequence of one-parameter families that smooth out one node at a time. The Jacobian of a curve of compact type is the product of the Jacobians of the irreducible components, and the Abel--Jacobi map is a product of the Abel--Jacobi maps on the components, with certain additional weights at the preimages of the nodes.

We now describe the Abel--Jacobi map for curves having one non-separating node. Let $(X,p_1,\ldots,p_n)$ be an irreducible curve having a single node. Let $\widetilde{X}$ be its normalization, and let $q^{\pm}\in \widetilde{X}$ be the preimages of the node. A line bundle of degree zero on $X$ is given by the data $(L,\xi)$ of a degree zero line bundle on $\widetilde{X}$ and a non-zero complex number $\xi\in \mathbb{C}^*$ that defines an isomorphism of the stalks at $q^{\pm}$. The compactified Jacobian $\overline{\op{Jac}}(X)$ is a semi-abelic variety obtained by letting this parameter tend to zero and infinity. Formally, the compactified Jacobian is obtained from the $\mathbb{P}^1$-bundle on $\op{Jac}(\widetilde{X}_0)$ by identifying  the points $(L,0)$ and $(L+\mathcal{O}_{\widetilde{X}}(q^+-q^-),\infty)$ for each $L\in \op{Jac}(\widetilde{X})$.

Now let $(X_t,p_1,\ldots,p_n)$ be a family of smooth marked curves degenerating to $X=X_0$. Choose a basis $A_i, B_i$ of $H_1(X_t,\ZZ)$ such that the degeneration corresponds to contracting the cycle $A_g$, and let $\omega_i$ be the basis for $H^0(X_t,\Omega)$ dual to the $A$-cycles. In the limit $t\rightarrow 0$, the differentials $\omega_1,\ldots,\omega_{g-1}$ degenerate to holomorphic differentials on $\widetilde{X}$ dual to $A_1,\ldots,A_{g-1}$, while the differential $\omega_g$ degenerates to a meromorphic differential on $\widetilde{X}$ having zero $A$-periods and having simple poles with residues $\pm 1/2\pi i$ at $q^{\pm}$. In other words, the limit of the first $g-1$ components of formula (\ref{eq:AJ}) is the Abel--Jacobi map $s_{\ud}$ of the normalization $(\widetilde{X},p_1,\ldots,p_n)$. The limit of the last component is a finite number, because the points $p_i$ are distinct from the $q^{\pm}$. The exponential of this number is the parameter $\xi$ that determines the gluing data of the line bundle on $X$ that is the limit of $s_{\ud}(X_t,p_1,\ldots,p_n)$:
\begin{equation}
\xi=\exp\left(\displaystyle\sum_{i=1}^n d_i \displaystyle\int_q^{p_i} \omega_g\right).
\label{eq:xi}
\end{equation}
In other words, the image of $(X,p_1,\ldots,p_n)$ under the Abel--Jacobi map $s_{\ud}$ is defined as the Abel--Jacobi map on the normalization plus the gluing parameter $\xi$ given by formula (\ref{eq:xi}) above:
\begin{equation}
s_{\ud}(X,p_1,\ldots,p_n)=(s_{\ud}(\widetilde{X},p_1,\ldots,p_n),\xi)\in \op{Jac}(\widetilde{X})\times \mathbb{C}^*\subset \overline{\op{Jac}}(X).
\label{eq:AJirred}
\end{equation}
Note that the parameter $\xi$ is always finite and non-zero, in other words this image always lies in the smooth locus of the compactified Jacobian.

Finally, suppose that $(X,p_1,\ldots,p_n)\in \overline{\calM}_{g,n}^o$ is a stable curve having only one non-separating node. Let this node be $q$, lying on an irreducible component $X_0$, so that the normalization at $q$ is a stable curve $(X',p_1,\ldots,p_n)$ of compact type and of genus $g-1$. The Abel--Jacobi map of $X$ is then equal to
\begin{equation}
s_{\ud}(X,p_1,\ldots,p_n)=(s_{\ud}(X',p_1,\ldots,p_n),\xi)\in \op{Jac}(X')\times \CC^*\subset\overline{\op{Jac}}(X).
\end{equation}
Here $s_{\ud}(X',p_1,\ldots,p_n)$ is the Abel--Jacobi map for a curve of compact type as described above, and the parameter $\xi$ is given by the same formula (\ref{eq:xi}), where, however, each $1$-form $\omega_i$ is non-zero on only one connected component of $X'$, and the path of integration consists only of the part of the path from $q$ to $p_i$ that lies on that component (and hence may be empty). Note that $\xi$ remains finite, hence the image of $s_{\ud}$ lies in the smooth locus of $\overline{\op{Jac}}(X)$.
\begin{proof}[Proof of Theorem~\ref{thm:Mg}] Let $s_\ud$ denote the Abel--Jacobi map $\overline{\calM}_{g,n}^o\to\calX_g'$ described above. The closure of the double ramification cycle $\overline{R}_\ud^o\subset\overline{\calM}_{g,n}^o$ is  the pullback of the zero section $s_\ud^*(z_g')$. The class of the zero section $Z_g'$ is given by a polynomial in $\Theta$, $D$ and $\Delta$ by Theorem~\ref{thm:main}, so to compute the class $[\overline{R}_\ud^o]$ we need to compute the pullbacks of $\Theta$, $D$, and $\Delta$ under $s_\ud$.

The pullback $s_\ud^*\T$ on $\overline{\calM}_{g,n}$ was computed by Hain in \cite{hainnormal}, and an alternative computation of it is one of the main results of \cite{grza1} (note also that a closely related divisor class was computed recently by M\"uller \cite{muller}, and a computation on the moduli space of curves with rational tails was done by Cavalieri, Marcus, and Wise in \cite{cavalieri}). This class is expressed in terms of the standard divisor classes on $\overline{\calM}_{g,n}$ in the following way:
$$
s_\ud^*\T=\frac{1}{2}\displaystyle\sum_{i=1}^n d_i^2 K_i-\frac{1}{2}\displaystyle\sum_{P\subseteq I,|P|\geq 2}
\left(d_P^2-\displaystyle\sum_{i\in P}d_i^2\right)\delta_0^P-\frac{1}{2}
\displaystyle\sum_{h> 0,P\subseteq I} d_P^2
\delta_h^P.
$$
Here $K_i$ denotes the pullback of the relative dualizing sheaf of the universal curve $\overline{\calM}_{g,1}\to \overline{\calM}_g$ under the projection map $\pi_i:\overline{\calM}_{g,n}\to\overline{\calM}_{g,1}$ forgetting all but the $i$-th marked point, $I=\{1,\ldots,n\}$ denotes the indexing set, $d_P=\sum_{i\in P} d_i$, and $\delta_h^P$ denotes the class of the boundary divisor whose generic point is a reducible curve consisting of a smooth genus $h$ component containing the marked points indexed by $P$ joined at a node to a smooth genus $g-h$ component containing the remaining marked points.

Let $\delta_{irr}$ denote the class of the boundary divisor whose generic point is an irreducible curve with a node. The preimage of $D$ is the locus of curves whose Jacobian is a semiabelic variety. Since $D$ is a pullback of the boundary of $\calA_g'$, the map $s_\ud:\delta_{irr}\to D$ factors through a lift of $\overline{\calM}_g^o\to\calA_g'$, and the multiplicity is thus one, so we have $s_\ud^*D=\delta_{irr}$.

Finally, the singular locus of the compactified Jacobian of a curve of geometric genus $g-1$ parameterizes torsion free, rank one, degree zero sheaves that are not line bundles. Equivalently, it is the singular locus of the corresponding semiabelic variety, i.e.~is the locus where this variety is non-normal, the image of the glued 0 and $\infty$ section. We have seen above that the image of $s_\ud$ on $\overline{\calM}_{g,n}^o$ is disjoint from this locus, and thus disjoint from $\Delta$; hence $s_{\ud}^*\Delta=0$, proving the theorem.
\end{proof}

The preimage of $\calA_g'$ under the Torelli map is the locus of all stable curves of geometric genus at least $g-1$, while we have shown above that the map $s_\ud$ extends to curves having at most one non-separating node --- and only computed the double ramification cycle on that locus. In the following examples we show that the Abel--Jacobi map $s_{\ud}$ does not in general extend to curves having two or more non-separating nodes, and thus $\overline{\calM}_{g,n}^o$ is the largest locus on which we can compute the double ramification cycle by pulling back the zero section from $\calX_g'$.

\begin{ex}\label{fails1} Assume that $g\geq 2$, and let $(X,p_1,\ldots,p_n)\in\overline{\calM}_{g,n}$ be the banana curve having two smooth components $X'$ and $X''$ of genera $h>0$ and $g-h-1$, respectively, with the points $q_1^+$ and $q_2^+$ on $X'$ glued respectively to $q_1^-$ and $q_2^-$ on $X''$. Assume for simplicity that the point $p_1$ is on $X'$, while the remaining points are on $X''$. The compactified Jacobian of $X$ is obtained from a $\mathbb{P}^1$-bundle over $\op{Jac}(X')\times \op{Jac}(X'')$ by identifying $(L',L'',0)$ and $(L'+\mathcal{O}_{X'}(q_1^+-q_2^+),L''+\mathcal{O}_{X''}(q_1^--q_2^-),\infty)$ for all $L'\in\op{Jac}(X')$ and $L''\in\op{Jac}(X'')$.

We can construct $X$ as a limit as $t\rightarrow 0$ of a family $X_t$ of irreducible nodal curves with one node, such that $q_1^+=q_1^-$ is the limit of the node of $X_t$, while $q_2^+=q_2^-$ is obtained by collapsing a homologically trivial cycle on the normalization $\widetilde{X}_t$. For each $X_t$, the Abel--Jacobi map is given by formula (\ref{eq:AJirred}), and hence the Abel--Jacobi map $s_{\ud}$ of $X$ should be the limit of (\ref{eq:AJirred}) as $t\rightarrow 0$. The limit of the family $\widetilde{X}_t$ is a curve of compact type with irreducible components $X'$ and $X''$ joined at $q^+_2=q^-_2$ but not at the other node. As $t\rightarrow 0$, the second component $\xi$ of (\ref{eq:AJirred}) may degenerate to infinity, hence the limit may lie in the singular locus of $\overline{\op{Jac}}(X)$. However, the limit of the first component $s_{\ud}(\widetilde{X}_t,p_1,\ldots,p_n)$ is a point on $\op{Jac}(X')\times \op{Jac}(X'')$ determined by formula (\ref{eq:AJcompact}). Hence, we see that $s_{\ud}(X,p_1,\ldots,p_n)$ should be a point on $\overline{\op{Jac}}(X)$ coming from a $\mathbb{P}^1$-fiber lying over the point in $\op{Jac}(X')\times \op{Jac}(X'')$ whose first coordinate is $\mathcal{O}_{X'}(d_1p_1-d_1q_2^+)$.

If we now exchange the roles of the two nodes, we see that $s_{\ud}(X,p_1,\ldots,p_n)$ should be a point on $\overline{\op{Jac}}(X)$ coming from the $\mathbb{P}^1$-fiber lying above $(\mathcal{O}_{X'}(d_1p_1-d_1 q_1^+),*)$. We now see that, unless $|d_1|\leq 1$ or $q_1^+-q_2^+$ is a torsion point of $\op{Jac}(X')$, these two limits cannot correspond to the same point in the compactified Jacobian $\overline{\op{Jac}}(X)$. Hence the Abel--Jacobi map $s_{\ud}(X,p_1,\ldots,p_n)$ is undefined.
\end{ex}

\begin{ex}\label{fails2} Even in genus one, the Abel--Jacobi map does not extend to all of $\overline{\calM}_{1,n}$ for $n\geq 3$. Indeed, let $X$ be a cycle of three rational components $X_1$, $X_2$ and $X_3$, with $p_1\in X_1$, $p_2\in X_2$ and the remaining points on $X_3$. We can obtain $X$ as a limit of a family of rational nodal curves by letting $p_1$ and $p_2$ tend to the node from different directions. The smooth locus of a rational nodal curve is identified with $\mathbb{C}^*$, which is also the Jacobian, and the Abel--Jacobi map $s_{\ud}$ is the product of the coordinates of the marked points $p_i$ raised to the corresponding powers $d_i$. It is easy to see that when $p_1$ approaches zero and $p_2$ approaches infinity, this product depends on the relative rates of approach. Hence $s_{\ud}$ is not defined for the limit curve $X$.
\end{ex}

\begin{rem}
In view of the examples above, computing completely the class of the closure $\overline{R}_\ud\subset\overline{\calM}_{g,n}$ appears to be a problem completely different in nature from the one that we study. Indeed, even if one could compute the class $\delta_g$ on some bigger partial toroidal compactification of $\calA_g$ (see Remark~\ref{rem:goingtoAg} about the difficulties of this), this still would not suffice, as the Abel-Jacobi map does not extend to all of $\overline{\calM}_{g,n}$ as explained in the examples above, and is in fact undefined on a locus of codimension two. Thus it seems impossible to describe $\overline{R}_\ud$ geometrically as a preimage of some locus on some compactification of $\calX_g$, as there is simply no map there, and thus to go beyond the locus $\overline{R}_\ud^o$, one would need to develop completely different methods to study the closure of the double ramification cycle, either by resolving the indeterminacy of the map $s_\ud$ or by describing its boundary points in some other way.
\end{rem}

\section*{Acknowledgments}

We thank Maksym Fedorchuk for discussions on semistable reduction, Richard Hain for useful discussions related to Eliashberg's problem, Klaus Hulek for pointing out the importance of automorphisms of universal families, Robin de Jong for related discussions on normal functions, and Claire Voisin for explanations about the decomposition theorem. We are also very grateful to Gerard van der Geer for comments on the Fourier transform, and for reading a draft version of this text and suggesting numerous valuable improvements.

We thank the referees for many valuable suggestions and comments, that led to various useful corrections and clarifications in the text.


\begin{thebibliography}{vdGM12}
\bibitem[Ale02]{alexeev}
\newblock  V.~Alexeev.
\newblock Complete moduli in the presence of semiabelian group action.
\newblock {\em Ann. of Math.}, 155(3):611--708, 2002.

\bibitem[AB11]{albr}
V.~Alexeev and A.~Brunyate.
\newblock Extending {T}orelli map to toroidal compactifications of {S}iegel
  space.
\newblock {\em Invent. Math.}, 188(1):175--196, 2011.

\bibitem[Be10]{2010Beauville}
A.~Beauville,
\newblock The action of $SL_2$ on abelian varieties,
\newblock {\em J. Ramanujan Math. Soc.}, 25 (2010), no. 3, 253-263.

\bibitem[CMW12]{cavalieri}
R.~Cavalieri, S.~Marcus, and J.~Wise.
\newblock Polynomial families of tautological classes on
  $\mathcal{M}_{g,n}^{rt}$.
\newblock {\em J.~Pure Appl.~Algebra}, 216(4):950--981, 2012.

\bibitem[DM91]{denmur}
C.~Deninger and J.~Murre.
\newblock Motivic decomposition of abelian schemes and the {F}ourier--Mukai transform.
\newblock {\em J. Reine Angew. Math.}, 422:201--219, 1991.

\bibitem[EvdG04]{ekvdgorder}
T.~Ekedahl and G.~van~der Geer.
\newblock The order of the top {C}hern class of the {H}odge bundle on the
  moduli space of abelian varieties.
\newblock {\em Acta Math.}, 192(1):95--109, 2004.

\bibitem[EvdG05]{ekvdgcycles}
T.~Ekedahl and G.~van~der Geer.
\newblock Cycles representing the top {C}hern class of the {H}odge bundle on
  the moduli space of abelian varieties.
\newblock {\em Duke Math. J.}, 129(1):187--199, 2005.

\bibitem[EGH10]{ergrhu2}
C.~Erdenberger, S.~Grushevsky, and K.~Hulek.
\newblock Some intersection numbers of divisors on toroidal compactifications
  of {${\mathcal A}_g$}.
\newblock {\em J. Algebraic Geom.}, 19:99--132, 2010.

\bibitem[EV02]{esvi}
H.~Esnault and E.~Viehweg.
\newblock Chern classes of {G}auss-{M}anin bundles of weight 1 vanish.
\newblock {\em {$K$}-Theory}, 26(3):287--305, 2002.

\bibitem[Fab97]{faberonedim}
C.~Faber.
\newblock A non-vanishing result for the tautological ring of {${\mathcal
  {M}}_g$}.
\newblock 1997.
\newblock preprint arXiv:math/9711219.

\bibitem[Fab99]{faberconjecture}
C.~Faber.
\newblock A conjectural description of the tautological ring of the moduli
  space of curves.
\newblock In {\em Moduli of curves and abelian varieties}, Aspects Math., E33,
  pages 109--129. Vieweg, Braunschweig, 1999.

\bibitem[FP03]{fapalambdag}
C.~Faber and R.~Pandharipande.
\newblock Hodge integrals, partition matrices, and the {$\lambda\sb g$}
  conjecture.
\newblock {\em Ann. of Math. (2)}, 157(1):97--124, 2003.

\bibitem[FP05]{fapabntaut}
C.~Faber and R.~Pandharipande.
\newblock Relative maps and tautological classes.
\newblock {\em J. Eur. Math. Soc. (JEMS)}, 7(1):13--49, 2005.

\bibitem[FSZ10]{FSZ}
C.~Faber, S.~Shadrin, D.~Zvonkine.
\newblock Tautological relations and the r-spin Witten conjecture.
\newblock {\em Ann. Sci. \'Ec. Norm. Sup\'er.}, 43(4):621--658, 2010.

\bibitem[Ful98]{fultonintersection}
W.~Fulton.
\newblock {\em Intersection theory}, volume~2 of {\em Ergebnisse der Mathematik
  und ihrer Grenzgebiete. 3. Folge. A Series of Modern Surveys in Mathematics}.
\newblock Springer-Verlag, Berlin, second edition, 1998.

\bibitem[vdG98]{vdgeerchowa3}
G.~van~der Geer.
\newblock The {C}how ring of the moduli space of abelian threefolds.
\newblock {\em J. Algebraic Geom.}, 7(4):753--770, 1998.

\bibitem[vdG99]{vdgeercycles}
G.~van~der Geer.
\newblock {\em Cycles on the moduli space of abelian varieties}, pages 65--89.
\newblock Aspects Math., E33. Vieweg, Braunschweig, 1999.

\bibitem[vdGM12]{vdgmoo}
G.~van~der Geer and B.~Moonen.
\newblock {\em Abelian Varieties}.
\newblock 2012.
\newblock draft version available at
  \texttt{http://staff.science.uva.nl/\~{}bmoonen/boek/BookAV.html}.

\bibitem[GV05]{grva}
T.~Graber and R.~Vakil.
\newblock Relative virtual localization and vanishing of tautological classes
  on moduli spaces of curves.
\newblock {\em Duke Math. J.}, 130(1):1--37, 2005.

\bibitem[GH11]{grhu2}
S.~Grushevsky and K.~Hulek.
\newblock Principally polarized semiabelic varieties of torus rank up to 3, and
  the {A}ndreotti-{M}ayer loci.
\newblock {\em Pure and Applied Mathematics Quarterly, special issue in memory
  of Eckart Viehweg}, 7:1309--1360, 2011.

\bibitem[GH12]{grhu1}
S.~Grushevsky and K.~Hulek.
\newblock The class of the locus of intermediate jacobians of cubic threefolds.
\newblock {\em Invent. Math.}, 190(1):119--168, 2012.

\bibitem[GL08]{grle}
S.~Grushevsky and D.~Lehavi.
\newblock Some intersections in the {P}oincar\'e bundle and the universal theta
  divisor on {$\overline{{\mathcal A}_ g}$}.
\newblock {\em Int. Math. Res. Not. IMRN}, (1):Art. ID rnm 128, 19, 2008.

\bibitem[GZ12]{grza1}
S.~Grushevsky and D.~Zakharov.
\newblock The double ramification cycle and the theta divisor.
\newblock Preprint arXiv:1206.7001,
\newblock {\em Proc. of the AMS}, to appear.

\bibitem[Hai13]{hainnormal}
R.~Hain.
\newblock Normal Functions and the Geometry of Moduli Spaces of Curves.
\newblock {\it Handbook of Moduli, volume I}, 527--578, edited by Gavril Farkas, Ian Morrison, International Press, 2013.

\bibitem[Ion02]{ionel}
E.-N. Ionel.
\newblock Topological recursive relations in {$H^{2g}({\mathcal M}_{g,n})$}.
\newblock {\em Invent. Math.}, 148(3):627--658, 2002.

\bibitem[KS03]{kesa}
S.~Keel and L.~Sadun.
\newblock Oort's conjecture for {${\mathcal A}_g\otimes{\mathbb C}$}.
\newblock {\em J. Amer. Math. Soc.}, 16(4):887--900, 2003.

\bibitem[Ku93]{1993Kuennemann}
K.~K\"unnemann.
\newblock A Lefschetz decomposition for Chow motives of abelian schemes,
\newblock {\em Invent. Math.}, 113 (1993), no. 1, 85-102.

\bibitem[Loo95]{loo}
E.~Looijenga.
\newblock On the tautological ring of {${\mathcal M}_g$}.
\newblock {\em Invent. Math.}, 121(2):411--419, 1995.

\bibitem[LL97]{lolu}
E.~Looijenga and V.~Lunts.
\newblock A {L}ie algebra attached to a projective variety.
\newblock {\em Invent. Math.}, 129(2):361--412, 1997.

\bibitem[MV12]{mevi}
M.~Melo and F.~Viviani.
\newblock Comparing perfect and 2nd {V}oronoi decompositions: the matroidal
  locus.
\newblock {\em Math. Ann.}, 354(4):1521--1554, 2012.


\bibitem[Moo11]{moonen}
B.~Moonen.
\newblock On the Chow motive of an abelian scheme with non-trivial
  endomorphisms.
\newblock 2011.
\newblock preprint arXiv:1110.4264 (version 2).

\bibitem[M{\"u}l12]{muller}
F.~M{\"u}ller.
\newblock The pullback of a theta divisor to {${\mathcal M}_{g,n}$}.
\newblock 2012.
\newblock preprint arXiv:1203.3102.

\bibitem[Mum83]{mumforddimag}
D.~Mumford.
\newblock On the {K}odaira dimension of the {S}iegel modular variety.
\newblock In {\em Algebraic geometry---open problems ({R}avello, 1982)}, volume
  997 of {\em Lecture Notes in Math.}, pages 348--375, Berlin, 1983. Springer.

\bibitem[Nam80]{namikawabook}
Y.~Namikawa.
\newblock {\em Toroidal compactification of {S}iegel spaces}, volume 812 of
  {\em Lecture Notes in Mathematics}.
\newblock Springer, Berlin, 1980.

\bibitem[PP13]{2013PandharipandePixton}
\newblock R.~Pandharipande, A.~Pixton.
\newblock {\em Relations in the tautological ring of the moduli space of curves.}
\newblock arXiv:1301.4561

\bibitem[PPZ13]{2013PandharipandePixtonZvonkine}
\newblock R.~Pandharipande, A.~Pixton and D.~Zvonkine.
\newblock {\em Relations on $\overline{\mathcal{M}}_{g,n}$ via $3$-spin structures.}
\newblock arXiv:1303.1043

\bibitem[PT12]{2012PetersenTommasi}
\newblock D.~Petersen, O.~Tommasi.
\newblock {\em The Gorenstein conjecture fails for the tautological ring of $\overline{\mathcal{M}}_{2,n}$.}
\newblock arXiv:1210.5761

\bibitem[Po96]{1996Polishchuk}
A.~Polishchuk.
\newblock Biextension, Weil representation on derived categories, and theta functions,
\newblock Ph. D. Thesis, Harvard Univ., 1996.

\bibitem[SB06]{shepherdbarron}
N.~Shepherd-Barron.
\newblock Perfect forms and the moduli space of abelian varieties.
\newblock {\em Invent. Math.}, 163(1):25--45, 2006.

\bibitem[Tho07]{thompson}
G.~Thompson.
\newblock Skew invariant theory of symplectic groups, pluri-{H}odge groups and
  3-manifold invariants.
\newblock {\em Int. Math. Res. Not. IMRN}, (15):Art. ID rnm048, 32, 2007.

\bibitem[Voi12a]{voisin}
C.~Voisin.
\newblock Chow rings and decomposition theorems for {K}3 surfaces and
  {C}alabi-{Y}au hypersurfaces.
\newblock {\em Geom. and Top.}, 16(1):433--473, 2012.

\bibitem[Voi12]{voisinnotes}
C.~Voisin.
\newblock Chow rings, decomposition of the diagonal and the topology of
  families.
\newblock 2012.
\newblock available at
  http://www.math.jussieu.fr/\~{}voisin/Articlesweb/weyllectures.pdf.



\end{thebibliography}
\end{document}